\documentclass[11pt]{amsart}
\usepackage{hyperref}
\usepackage{amsfonts}
\usepackage{amsmath}
\usepackage{amssymb}
\usepackage{amscd}
\usepackage{graphicx}
\usepackage{latexsym}
\usepackage{color}
\usepackage{epsfig}
\usepackage{amsopn}
\usepackage{mathrsfs}
\usepackage{array}
\usepackage{xspace}
\usepackage{comment}

\theoremstyle{plain}
\newtheorem{lemma}{Lemma}[section]
\newtheorem*{theorem*}{Theorem}
\newtheorem*{lemma*}{Lemma}
\newtheorem*{proposition*}{Proposition}
\newtheorem*{conjecture*}{Conjecture}
\newtheorem*{corollary*}{Corollary}
\newtheorem*{problem*}{Problem}
\newtheorem{theorem}[lemma]{Theorem}

\newtheorem{corollary}[lemma]{Corollary}
\newtheorem{proposition}[lemma]{Proposition}

\theoremstyle{definition}
\newtheorem{definition}[lemma]{Definition}

\newtheorem{remark}[lemma]{Remark}

\DeclareMathOperator{\Coh}{Coh}
\DeclareMathOperator{\Ext}{Ext}
\DeclareMathOperator{\ext}{ext}

\DeclareMathOperator{\Amp}{Amp}

\DeclareMathOperator{\Stab}{Stab}
\DeclareMathOperator{\Nef}{Nef}
\DeclareMathOperator{\NE}{NE}
\DeclareMathOperator{\ch}{ch}
\DeclareMathOperator{\rk}{rk}
\DeclareMathOperator{\Pic}{Pic}
\DeclareMathOperator{\Bl}{Bl}
\DeclareMathOperator{\CrDiv}{CrDiv}

\newcommand{\Z}{\mathbb{Z}\xspace}

\newcommand{\Q}{\mathbb{Q}\xspace}
\newcommand{\cA}{\mathcal{A}\xspace}
\newcommand{\num}{\mathrm{num}}

\newcommand{\OO}{\mathcal{O}\xspace}

\newcommand{\R}{\mathbb{R}\xspace}
\newcommand{\C}{\mathbb{C}\xspace}

\newcommand{\cF}{\mathcal{F}\xspace}

\newcommand{\cT}{\mathcal{T}\xspace}
\newcommand{\cE}{\mathcal{E}\xspace}
\newcommand{\bP}{\mathbb{P}\xspace}
\newcommand{\PP}{\mathbb{P}\xspace}

\def\into{\ensuremath{\hookrightarrow}}
\def\onto{\ensuremath{\twoheadrightarrow}}

\begin{document}

\title{Nef cones of Hilbert schemes of points on surfaces}

\begin{abstract}
Let $X$ be a smooth projective surface of irregularity $0$.  The Hilbert scheme $X^{[n]}$ of $n$ points on $X$ parameterizes zero-dimensional subschemes of $X$ of length $n$.  In this paper, we discuss general methods for studying the cone of ample divisors on $X^{[n]}$.  We then use these techniques to compute the cone of ample divisors on $X^{[n]}$ for several surfaces where the cone was previously unknown.  Our examples include families of surfaces of general type and del Pezzo surfaces of degree $1$.  The methods rely on Bridgeland stability and the Positivity Lemma of Bayer and Macr\`i.
\end{abstract}

\thanks{J. Huizenga was partially supported by an NSF Mathematical Sciences Postdoctoral Research Fellowship.  B. Schmidt is partially supported by NSF grant DMS-1523496 (PI Emanuele Macr\`i) and a Presidential Fellowship of the Ohio State University.}

\author[B. Bolognese]{Barbara Bolognese}
\address{Department of Mathematics, Northeastern University, Boston, MA 02215}
\email{bolognese.b@husky.neu.edu}
\urladdr{}

\author[J. Huizenga]{Jack Huizenga}
\address{Department of Mathematics, The Pennsylvania State University, University Park, PA 16802}
\email{huizenga@psu.edu}
\urladdr{http://www.personal.psu.edu/jwh6013/}

\author[Y. Lin]{Yinbang Lin}
\address{Department of Mathematics, Northeastern University, Boston, MA 02115}
\email{lin.yinb@husky.neu.edu}
\urladdr{}

\author[E. Riedl]{Eric Riedl}
\address{Department of Mathematics, Statistics, and CS, University of Illinois at Chicago, Chicago, IL 60607}
\email{ebriedl@uic.edu}
\urladdr{}

\author[B. Schmidt]{Benjamin Schmidt}
\address{Department of Mathematics, The Ohio State University, Columbus, OH 43210}
\email{schmidt.707@osu.edu}
\urladdr{https://people.math.osu.edu/schmidt.707/}

\author[M. Woolf]{Matthew Woolf}
\address{Department of Mathematics, Statistics, and CS, University of Illinois at Chicago, Chicago, IL 60607}
\email{mwoolf@math.uic.edu}
\urladdr{http://people.uic.edu/~mwoolf/}

\author[X. Zhao]{Xiaolei Zhao}
\address{Department of Mathematics, Northeastern University, Boston, MA 02215}
\email{x.zhao@neu.edu}
\urladdr{}

\date{\today}


\maketitle
\setcounter{tocdepth}{1}
\tableofcontents

\section{Introduction}

If $X$ is a projective variety, the cone $\Amp(X)\subset N^1(X)$ of ample divisors controls the various projective embeddings of $X$.  It is one of the most important invariants of $X$, and carries detailed information about the geometry of $X$.  Its closure is the \emph{nef cone} $\Nef(X)$, which is dual to the Mori cone of curves (see for example \cite{Lazarsfeld}).  In this paper, we will study the nef cone of the Hilbert scheme of points $X^{[n]}$, where $X$ is a smooth projective surface over $\C$.

Nef divisors on Hilbert schemes of points on surfaces $X^{[n]}$ are sometimes easy to construct by classical methods.  If $L$ is an $(n-1)$-very ample line bundle on $X$, then for any $Z\in X^{[n]}$ we have an inclusion $H^0(L\otimes I_Z) \to H^0(L)$ which defines a morphism from $X^{[n]}$ to the Grassmannian $G(h^0(L)-n,h^0(L))$.  The pullback of an ample divisor on the Grassmannian is nef on $X^{[n]}$.  It is frequently possible to construct extremal nef divisors by this method.  For example, this method completely computes the nef cone of $X^{[n]}$ when $X$ is a del Pezzo surface of degree $\geq 2$ or a Hirzebruch surface (see \cite{ABCH}, \cite{BertramCoskun}).
Unfortunately, this approach to computing the nef cone is insufficient in general.  At the very least, to study nef cones of more interesting surfaces it would be necessary to study an analog of $k$-very ampleness for higher rank vector bundles, which is considerably more challenging than line bundles.

More recently, many nef cones have been computed by making use of Bridgeland stability conditions and the Positivity Lemma of Bayer and Macr\`i (see \cite{Bri07}, \cite{Bri08}, \cite{AB13}, and \cite{BayerMacri} for background on these topics, which will be reviewed in Section \ref{sec-prelim}).  Let ${\bf v} = \ch(I_Z)\in K_0(X)$, where $Z\in X^{[n]}$.  In the stability manifold $\Stab(X)$ for $X$ there is an open \emph{Gieseker chamber} $\mathcal{C}$ such that if $\sigma\in \mathcal{C}$ then $M_\sigma({\bf v})\cong X^{[n]}$, where $M_{\sigma}({\bf v})$ is the moduli space of $\sigma$-semistable objects with invariants $\bf v$.  The Positivity Lemma associates to any $\sigma \in \overline{\mathcal C}$ a nef divisor on $X^{[n]}$.  Stability conditions in the boundary $\partial \mathcal C$ frequently give rise to extremal nef divisors.  The Positivity Lemma also classifies the curves orthogonal to a nef divisor constructed in this way, and so gives a tool for checking extremality.

The stability manifold is rather large in general, so computation of the full Gieseker chamber can be unwieldy.  We deal with this problem by focusing on a small \emph{slice} of the stability manifold parameterized by a half-plane.    Up to scale, the corresponding divisors in $N^1(X^{[n]})$ form an affine ray.
The nef cone $\Nef(X^{[n]})$ is spanned by a codimension 1 subcone identified with $\Nef(X)$ and other more interesting classes which are positive on curves contracted by the Hilbert--Chow morphism.  Since $\Nef(X^{[n]})$ is convex, we can study $\Nef(X^{[n]})$ by looking at positivity properties of divisors along rays in $N^1(X^{[n]})$ starting from a class in $\Amp(X)\subset \Nef(X^{[n]})$.  The  Positivity Lemma gives us an effective criterion for testing when divisors along the ray are nef.

The slices of the stability manifold that we consider are given by a pair of divisors $(H,D)$ on $X$ with $H$ ample and $-D$ effective.  The following is a weak version of one of our main theorems.
\begin{theorem}
Let $X$ be a smooth projective surface.  If $n\gg 0$, then there is an extremal nef divisor on $X^{[n]}$ coming from the $(H,D)$-slice.  It can be explicitly computed if both the intersection pairing on $\Pic(X)$ and the set of effective classes in $\Pic(X)$ are known.  An orthogonal curve class is given by $n$ points moving in a $g_n^1$ on a curve of a particular class.
\end{theorem}
See Section \ref{sec-nef} for more explicit statements, especially Corollary \ref{cor:Dbest} and Theorem \ref{thm:asymptotic}.  Stronger statements can also be shown under strong assumptions on $\Pic(X)$; for example, we study the Picard rank one case in detail in Section \ref{sec-generalType}.  Recall that if $X$ is surface of irregularity $q:=H^1(\OO_X)=0$ then $N^1(X^{[n]})$ is spanned by the divisor $B$ of nonreduced schemes and divisors $L^{[n]}$ induced by divisors $L\in \Pic(X)$; see Section \ref{ssec-divisorsAndCurves} for details.

\begin{theorem}\label{thm:introPic1}
Let $X$ be a smooth projective surface with $\Pic X \cong \Z H$, where $H$ is an ample divisor.  Let $a>0$ be the smallest integer such that $aH$ is effective.  If $$n \geq \max \{ a^2H^2,p_a(aH)+1\},$$ then $\Nef(X^{[n]})$ is spanned by the divisor $H^{[n]}$ and the divisor \begin{equation}\label{pic1divisor}\tag{$\ast$} \frac{1}{2} K_X^{[n]} + \left(\frac{a}{2}+\frac{n}{aH^2}\right)H^{[n]}-\frac{1}{2}B.\end{equation} An orthogonal curve class is given by letting $n$ points move in a $g_n^1$ on a curve in $X$ of class $aH$.
\end{theorem}

Note that in the Picard rank $1$ case the divisor class (\ref{pic1divisor}) is frequently of the form $\lambda H^{[n]}  - \frac{1}{2}B$ for a non-integer number $\lambda\in \Q$.  Any divisor constructed from an $(n-1)$-very ample line bundle will be of the form $\lambda H^{[n]} - \frac{1}{2}B$ with $\lambda \in \Z$, so in general the edge of the nef cone cannot be obtained from line bundles in this way.  

The required lower bound on $n$ in Theorem \ref{thm:introPic1} can be improved in specific examples where special linear series on hyperplane sections are better understood.

\begin{theorem}
Let $X$ be one of the following surfaces:
\begin{enumerate}
\item\label{surfaceCase} a very general hypersurface in $\PP^3$ of degree $d\geq 4$, or
\item\label{branchedCase} a very general degree $d$ cyclic branched cover of $\PP^2$ of general type.
\end{enumerate}
In either case, $\Pic(X) \cong \Z H$ with $H$ effective.  Suppose $n\geq d-1$ in the first case, and $n\geq d$ in the second case.  Then $\Nef(X^{[n]})$ is spanned by $H^{[n]}$ and the divisor class (\ref{pic1divisor}) with $a=1$.
\end{theorem}

Finally, in Section \ref{sec-delPezzo} we compute the nef cone of $X^{[n]}$ where $X$ is a smooth del Pezzo surface of degree $1$ and $n\geq 2$ is arbitrary.  This computation was an open problem posed by Bertram and Coskun in \cite{BertramCoskun}; they noted that the method of $k$-very ample line bundles would not be sufficient to prove the expected answer.  Since $X$ has Picard rank $9$, this computation makes full use of the general methods developed in Section \ref{sec-nef}.  If $C\subset X$ is a reduced, irreducible curve which admits a $g_n^1$, we write $C_{[n]}$ for the curve in the Hilbert scheme $X^{[n]}$ given by letting $n$ points move in a $g_n^1$ on $C$.

\begin{theorem}
Let $X$ be a smooth del Pezzo surface of degree $1$.  The Mori cone of curves $\NE(X^{[n]})$ is spanned by the $240$ classes $E_{[n]}$ given by $(-1)$-curves $E\subset X$, the class of a curve contracted by the Hilbert--Chow morphism, and the class $F_{[n]}$, where $F\in |{-K_X}|$ is an anticanonical curve.  The nef cone is determined by duality.
\end{theorem}

Many previous authors have used Bridgeland stability conditions to study nef cones and wall-crossing for Hilbert schemes $X^{[n]}$ and moduli spaces of sheaves $M_H({\bf v})$ for various classes of surfaces.  For instance, the program was studied for $\PP^2$ in \cite{ABCH}, \cite{CoskunHuizenga2}, \cite{BerMarWan}, and \cite{LiZhao}, for Hirzebruch and del Pezzo surfaces in \cite{BertramCoskun}, abelian surfaces in \cite{YanagidaYoshioka} and \cite{MM}, K3 surfaces in \cite{BayerMacri}, \cite{BayerMacri2} and \cite{HassettTschinkel}, and Enriques surfaces in \cite{Nuer}.  Our results unify several of these approaches.  Additionally, nef cones were classically studied in the context of $k$-very ample line bundles in papers such as \cite{EGH}, \cite{BelSomGot}, \cite{BelFraSom}, and \cite{CatGot}.

\subsection*{Acknowledgements}
This work was initiated at the 2015 Algebraic Geometry Bootcamp preceding the Algebraic Geometry Summer Research Institute organized by the AMS and the University of Utah. We would like to thank the organizers of both programs for providing the wonderful environment where this collaboration could happen.  Additionally, we would like to thank Arend Bayer, Izzet Coskun, and Emanuele Macr\`i for many valuable discussions on Bridgeland stability.

\section{Preliminaries}
\label{sec-prelim}

Throughout the paper, we let $X$ be a smooth projective surface over $\C$.

\subsection{Divisors and curves on $X^{[n]}$}\label{ssec-divisorsAndCurves} For simplicity we assume that $X$ has irregularity $q=h^1(\OO_X)=0$ in this subsection.  By work of Fogarty \cite{Fogarty}, the Hilbert scheme $X^{[n]}$ is a smooth projective variety of dimension $2n$ which resolves the singularities in the symmetric product $X^{(n)}$ via the Hilbert--Chow morphism $X^{[n]}\to X^{(n)}$.  A line bundle $L$ on $X$ induces the $S_n$-equivariant line bundle $L^{\boxtimes n}$ on $X^n$ which descends to a line bundle $L^{(n)}$ on the symmetric product $X^{(n)}$.  The pullback of $L^{(n)}$ by the Hilbert--Chow morphism $X^{[n]}\to X^{(n)}$ defines a line bundle on $X^{[n]}$ which we will denote by $L^{[n]}$.  Intuitively, if $L\cong \OO_X(D)$ for a reduced effective divisor $D\subset X$, then $L^{[n]}$ can be represented by the divisor $D^{[n]}$ of schemes $Z\subset X$ which meet $D$.

Fogarty shows that $$\Pic(X^{[n]})\cong \Pic(X) \oplus \Z (B/2),$$ where $\Pic(X) \subset \Pic(X^{[n]})$ is embedded by $L\mapsto L^{[n]}$ and $B$ is the locus of non-reduced schemes, i.e., the exceptional divisor of the Hilbert--Chow morphism \cite{Fogarty2}.  Tensoring by the real numbers, the Neron--Severi space $N^1(X^{[n]})$ is therefore spanned by $N^1(X)$ and $B$.

There are also curve classes in $X^{[n]}$ induced by curves in $X$.  Two different constructions are immediate.  Let $C\subset X$ be a reduced and irreducible curve.  
\begin{enumerate}
\item There is a curve $\tilde C_{[n]}$ in $X^{[n]}$ given by fixing $n-1$ general points of $X$ and letting an $n$th point move along $C$.  
\item If $C$ admits a $g_n^1$, i.e., a degree $n$ map to $\PP^1$, then the fibers of $C\to \PP^1$ give a rational curve $\PP^1\to X^{[n]}$.  We write $C_{[n]}$ for this class.
\end{enumerate}
These constructions preserve intersection numbers, in the sense that if $D\subset X$ is a divisor and $C\subset X$ is a curve then $$D^{[n]}\cdot \tilde C_{[n]} = D^{[n]}\cdot C_{[n]} = D\cdot C.$$ 

Part of the nef cone $\Nef(X^{[n]})$ is easily described in terms of the nef cone of $X$.  If $D$ is an ample divisor, then $D^{(n)}$ is ample so $D^{[n]}$ is nef.  In the limit, we find that if $D$ is nef then $D^{[n]}$ is nef.  Conversely, if $D$ is not nef then there is an irreducible curve  $C$ with $D\cdot C<0$, so $D^{[n]}\cdot \tilde{C}_{[n]}<0$ and $D^{[n]}$ is not nef. Under the Fogarty isomorphism, $$\Nef(X^{[n]}) \cap N^1(X) = \Nef(X).$$  The hyperplane $N^1(X)\subset N^1(X^{[n]})$ is orthogonal to any curve contracted by the Hilbert--Chow morphism, so all the divisors in $\Nef(X)\subset \Nef(X^{[n]})$ are extremal.  Since $B$ is the exceptional locus of the Hilbert--Chow morphism, we see that any nef class must have non-positive coefficient of $B$.  After scaling, then, we see that computation of the cone $\Nef(X^{[n]})$  reduces to describing the nef classes of the form $L^{[n]}-\frac{1}{2}B$ lying outside $\Nef(X)\subset \Nef(X^{[n]})$.

\subsection{Bridgeland stability conditions}
We now recall some basic definitions and properties of Bridgeland stability conditions. We fix a polarization $H \in \Pic(X)_{\R}$. For any divisor $D \in \Pic(X)_{\R}$ the twisted Chern character $\ch^{D} = e^{-D} \ch$ can be expanded as
\begin{align*}
\ch_0^{D} &= \ch_0, \\
\ch_1^{D} &= \ch_1 - D  \ch_0, \\
\ch_2^{D} &= \ch_2 - D \cdot \ch_1 + \frac{D^2}{2} \ch_0.
\end{align*}

Recall that a Bridgeland stability condition is a pair $\sigma = (Z, \cA)$ where $Z: K_0(X) \to \C$ is an additive homomorphism and $\cA \subset D^b(X)$ is the heart of a bounded t-structure. In particular, $\cA$ is an abelian category. Moreover, $Z$ maps any non trivial object in $\cA$ to the upper half plane or the negative real line. The $\sigma$-slope function is defined by
\[\nu_\sigma = -\frac{\Re Z}{\Im Z},\]
and $\sigma$-(semi)stability of objects of $\cA$ is defined in terms of this slope function.  More technical requirements are the existence of Harder--Narasimhan filtrations and the support property. We recommend Bridgeland's article \cite{Bri07} for a more precise definition. The support property is well explained in Appendix A of \cite{BMS14}.

In the case of surfaces, Bridgeland \cite{Bri08} and Arcara--Bertram \cite{AB13} showed how to construct Bridgeland stability conditions in a slice corresponding to a choice of an ample divisor $H\in \Pic(X)_\R$ and arbitrary twisting divisor $D\in \Pic(X)_\R$. The classical Mumford slope function for twisted Chern characters is defined by
\[\mu_{H,D} = \frac{H \cdot\ch_1^{D}}{H^2 \ch_0^{D}},\]
where torsion sheaves are interpreted as having positive infinite slope.
Given a real number $\beta \in \R$ there are two categories defined as
\begin{align*}
\cT_{\beta} &= \{E \in \Coh(X) : \text{any quotient $E \onto G$
satisfies $\mu_{H,D}(G) > \beta$} \}, \\
\cF_{\beta} &=  \{E \in \Coh(X) : \text{any subsheaf $F \into E$
satisfies $\mu_{H,D}(F) \leq \beta$} \}.
\end{align*}
A new heart of a bounded t-structure is defined as the extension
closure $\cA_{\beta} := \langle \cF_{\beta}[1], \cT_{\beta} \rangle$. We fix an additional positive real number $\alpha$ and define the homomorphism as
\[Z_{\beta,\alpha} = -\ch^{D + \beta H}_2 + \frac{\alpha^2 H^2}{2} \ch_0^{D + \beta H} + i H \cdot \ch_1^{D + \beta H}.\]  The pair $\sigma_{\beta,\alpha} := (Z_{\beta,\alpha},\cA_{\beta})$ is then a Bridgeland stability condition.  The $(H,D)$-\emph{slice} of stability conditions is the family of stability conditions $\{\sigma_{\beta,\alpha}:\beta,\alpha\in \R, \alpha>0\}$ parameterized by the $(\beta,\alpha)$ upper half plane.

\begin{definition} Fix a set of invariants ${\bf v}\in K_{0}(X)$.
\begin{enumerate} 
\item Let ${\bf w}\in K_0(X)$ be a vector such that ${\bf v}$ and ${\bf w}$ do not have the same $\sigma_{\beta,\alpha}$-slope everywhere in the $(H,D)$-slice.  The \textit{numerical wall} for $\bf v$ given by $\bf w$ is the set of points $(\beta, \alpha)$ where $\bf v$ and $\bf w$ have the same $\sigma_{\beta,\alpha}$-slope.
\item A numerical wall for $\bf v$ given by a vector $\bf w$ as above is a \emph{wall} (or \emph{actual wall}) if there is a point $(\beta,\alpha)$ on the wall and an exact sequence $0\to F\to E \to G\to 0$ in $\cA_\beta$, where $\ch F = {\bf w}$, $\ch E = {\bf v}$, and $F,E,G$ are $\sigma_{\beta,\alpha}$-semistable objects (of the same $\sigma_{\beta,\alpha}$-slope).
\end{enumerate}
\end{definition}

We write $K_{\num}(X)$ for the numerical Grothendieck group of classes in $K_0(X)$ modulo numerical equivalence. Note that numerical walls for ${\bf v}\in K_0(X)$ only depend on the numerical class of ${\bf v}$, while actual walls a priori depend on $c_1({\bf v})\in \Pic(X)$.  The structure of walls in a slice is heavily restricted by Bertram's Nested Wall Theorem. This was first observed for Picard rank one with $D = 0$, but the proof immediately generalizes by replacing $\ch$ by $\ch^D$ everywhere.

\begin{theorem}[\cite{Mac14}]
Let ${\bf v}\in K_{0}(X)$. 

\begin{enumerate}
\item  Numerical walls for $\bf v$ can either be semicircles with center on the $\beta$-axis or the unique vertical line given by $\beta = \mu_{H,D}({\bf v})$. Moreover, the apex of each semicircle lies on the hyperbola $\Re Z_{\beta,\alpha}({\bf v}) = 0$.

\item  Numerical walls for $\bf v$ are disjoint, and the semicircular walls on either side of the vertical wall are nested.

\item  If $W_1$ and $W_2$ are two semicircular numerical walls left of the vertical wall with centers $(s_{W_1},0)$ and $(s_{W_2},0)$, then $W_2$ is nested inside $W_1$ if and only if $s_{W_1}<s_{W_2}$. 

\item Suppose $0\to F\to E\to G\to 0$ is an exact sequence destabilizing an object $E$ with $\ch(E) = {\bf v}$ at a point $(\beta,\alpha)$ on a numerical wall $W$, in the sense that all three objects have the same $\sigma_{\beta,\alpha}$-slope and this is an exact sequence in $\cA_{\beta}$.  Then it is an exact sequence of objects in $\cA_{\beta'}$ with the same $\sigma_{\beta',\alpha'}$-slope for all $(\beta',\alpha')\in W$. That is, $E$ is destabilized along the entire wall. 
\end{enumerate}
\end{theorem}

\subsection{Slope and discriminant} The explicit geometry of walls is frequently best understood in terms of slopes and discriminants; the formulas presented here previously appeared in \cite{CoskunHuizenga} in the context of $\PP^2$.  When the rank is nonzero, we define
$$
\Delta_{H,D} = \frac{1}{2} \mu_{H,D}^2 - \frac{\ch_2^D}{H^2 \ch_0^D}.
$$
The Bogomolov inequality gives $\Delta_{H,D}(E)\geq 0$ whenever $E$ is an $(H,D)$-twisted Giesker semistable sheaf.  Observe that $\Delta_{H,D+\beta H} = \Delta_{H,D}$ for every $\beta\in \R$.  A straightforward calculation shows that for vectors of nonzero rank the slope function for the stability condition $\sigma_{\beta,\alpha}$ in the $(H,D)$-slice is given by \begin{equation}\label{slopeFormula}\nu_{\sigma_{\beta,\alpha}} = \frac{(\mu_{H,D}-\beta)^2-\alpha^2-2\Delta_{H,D}}{(\mu_{H,D}-\beta)}\end{equation}

Suppose ${\bf v}, {\bf w}$ are two classes with positive rank, and let their slopes and discriminants be $\mu_{H,D},\Delta_{H,D}$ and $\mu'_{H,D},\Delta'_{H,D}$, respectively.  The numerical wall $W$ in the $(H,D)$-slice where $\bf v$ and $\bf w$ have the same slope is computed as follows.

\begin{itemize}
\item If $\mu_{H,D} = \mu_{H,D}'$ and $\Delta_{H,D}=\Delta_{H,D}'$, then $\bf v$ and $\bf w$ have the same slope everywhere in the slice, so there is no numerical wall.

\item If $\mu_{H,D} = \mu_{H,D}'$ and $\Delta_{H,D} \neq \Delta_{H,D}'$, then $W$ is the vertical wall $\beta = \mu_{H,D}$.

\item If $\mu_{H,D} \neq \mu_{H,D}'$, then Equation (\ref{slopeFormula}) implies $W$ is the semicircle with center $(s_W,0)$ and radius $\rho_W$, where
\begin{align}
\label{centerFormula}s_W &= \frac{1}{2} (\mu_{H,D}+\mu'_{H,D}) - \frac{\Delta_{H,D}-\Delta'_{H,D}}{\mu_{H,D}-\mu'_{H,D}}, \\
\label{rhoFormula} 
\rho_W^2 &= (s_W-\mu_{H,D})^2 - 2\Delta_{H,D}
\end{align}
provided that the expression defining $\rho_W^2$ is positive; if it is negative then the wall is empty. 
\end{itemize}

 Notice that if $\Delta_{H,D}({\bf v})\geq 0$ then numerical walls for $\bf v$ left of the vertical wall accumulate at the point \begin{equation}\label{accumulationPoint}\left(\mu_{H,D}({\bf v})-\sqrt{2 \Delta_{H,D}({\bf v})},0\right)\end{equation} as their radii go to $0$. 

\subsection{Nef divisors and the Positivity Lemma} 

In this section, we describe the Positivity Lemma of Bayer and Macr\`{i}. Let $\sigma = (Z, \cA)$ be a stability condition on $X$, ${\bf v} \in K_{\num}(X)$ and $S$ a proper algebraic space of finite type over $\C$. Let $\cE \in D^b(X \times S)$ be a flat family of $\sigma$-semistable objects of class ${\bf v}$, i.e., for every $
\C$-point $p \in S$, the derived restriction $\cE|_{\pi_S^{-1}(\{p\})}$ is $\sigma$-semistable of class ${\bf v}$. Then Bayer and Macr\`{i} define a numerical divisor class $D_{\sigma,\cE}\in N^1(S)$ on the space $S$ by assigning its intersection with any projective integral curve $C \subset S$:
\[D_{\sigma,\cE} \cdot C = \Im \left(  -\frac{Z((p_X)_* \cE |_{C \times X})}{Z({\bf v})} \right).\] The Positivity Lemma shows that this divisor inherits positivity properties from the homomorphism $Z$, and classifies the curve classes orthogonal to the divisor. Recall that two $\sigma$-semistable objects are $S$-{\it equivalent} with respect to $\sigma$ if their sets of Jordan--H\"older factors are the same.

\begin{theorem}[{Positivity Lemma, \cite[Lemma 3.3]{BayerMacri}}]
The divisor $D_{\sigma,\cE}\in N^1(S)$ is nef. Moreover, if $C\subset S$ is a projective integral curve then $D_{\sigma,\cE} \cdot C = 0$ if and only if two general objects parameterized by $C$ are $S$-equivalent with respect to $\sigma$.
\end{theorem}

Our primary use of the Positivity Lemma is to attempt to construct extremal nef divisors on Hilbert schemes of points.  Thus it is important to recover Hilbert schemes of points as Bridgeland moduli spaces.  Recall that a torsion-free coherent sheaf $E$ is \emph{$(H,D)$-twisted Gieseker semistable} if for every $F\subset E$ we have $$\frac{\chi(F\otimes \OO_X(mH-D))}{\rk(F)}\leq \frac{\chi(E\otimes \OO_X(mH-D))}{\rk(E)}$$  for all $m\gg 0$, where the Euler characteristic is computed formally via Riemann--Roch; see \cite{MatsukiWentworth}.  For any class ${\bf v}\in K_0(X)$, there are projective moduli spaces $M_{H,D}({\bf v})$ of $S$-equivalence classes of $(H,D)$-twisted Gieseker semistable sheaves with class ${\bf v}$.  If ${\bf v} = (1,0,-n)$ is the Chern character of an ideal sheaf of $n$ points then $M_{H,D}({\bf v}) = X^{[n]}$.  Note that if the irregularity of $X$ is nonzero, then it is crucial to fix the determinant.  

Fix an $(H,D)$-slice in the stability manifold, and fix a vector ${\bf v}\in K_{0}(X)$ with positive rank.  If $\beta$ lies to the left of the vertical wall $\beta = \mu_{H,D}({\bf v})$ for ${\bf v}$, then for $\alpha \gg 0$ the moduli space coincides with a twisted Gieseker moduli space.

\begin{proposition}[The large volume limit {\cite{Bri08, Mac14}}]
\label{prop:large_volume_limit}
Fix divisors $(H,D)$ giving a slice in $\Stab(X)$.  Let ${\bf v}\in K_{0}(X)$ be a vector with positive rank, and let $\beta\in \R$ be such that $\mu_{H,D}({\bf v})>\beta$. If $E \in \cA_\beta$ has $\ch(E) = {\bf v}$ then $E$ is $\sigma_{\beta,\alpha}$-semistable for all $\alpha \gg 0$ if and only if $E$ is an $(H,D-\tfrac{1}{2}K_X)$-twisted Gieseker semistable sheaf.

Moreover, in the quadrant of the $(H,D)$-slice left of the vertical wall there is a largest semicircular wall for ${\bf v}$, called the Gieseker wall.  For all $(\beta,\alpha)$ between this wall and the vertical wall, the moduli space $M_{\sigma_{\beta,\alpha}}({\bf v})$ coincides with the moduli space $M_{H,D-K_X/2}({\bf v})$ of $(H,D-\tfrac{1}{2}K_X)$-twisted Gieseker semistable sheaves.
\end{proposition}

We use these results as follows.  Let ${\bf v} = (1,0,-n)\in K_0(X)$ be the vector for the Hilbert scheme $X^{[n]}$, and let $\sigma_{+}$ be a stability condition in the $(H,D)$-slice lying above the Gieseker wall, so that $M_{\sigma_+}({\bf v}) \cong X^{[n]}$.  Let $\cE/(X\times X^{[n]})$ be the universal ideal sheaf, and let $\sigma_0$ be a stability condition on the Gieseker wall.  By the definition of the Gieseker wall, $\cE$ is a family of $\sigma_0$-semistable objects, so there is an induced nef divisor $D_{\sigma_0,\cE}$ on $X^{[n]}$.  Furthermore, curves orthogonal to $D_{\sigma_0,\cE}$ are understood in terms of destabilizing sequences along the wall, so it is possible to test for extremality.

\section{Gieseker walls and the nef cone}
\label{sec-nef}

Fix an ample divisor $H \in \Pic(X)$ with $H^2=d$ and an \emph{antieffective} divisor $D$.  In this section we study the nef divisor arising from the Gieseker wall (i.e., the largest wall where some ideal sheaf is destabilized) in the slice of the stability manifold given by the pair $(H,D)$.  We first compute the Gieseker wall, and then investigate when the corresponding nef divisor is in fact extremal.

\subsection{Bounding higher rank walls}

The main difficulty in computing extremal rays of the nef cone is to show that a destabilizing subobject along the Gieseker wall is a line bundle, and not some higher rank sheaf.  We first prove a lemma which generalizes \cite[Proposition 8.3]{CoskunHuizenga2} from $X=\PP^2$ to an arbitrary surface.  We prove the result in slightly more generality than we will need here as we expect it to be useful in future work.

\begin{lemma}
\label{lem:higherRank}
Let $\sigma_0$ be a stability condition in the $(H,D)$-slice, and suppose
\[0\to F\to E\to G\to 0\]
is an exact sequence of $\sigma_0$-semistable objects of the same $\sigma_0$-slope, where $E$ is an $(H,D)$-twisted Gieseker semistable torsion-free sheaf. If the map $F\to E$ of sheaves is not injective, then the radius $\rho_W$ of the wall $W$ defined by this sequence satisfies
\[\rho_W^2 \leq \frac{(\min\{\rk (F) -1, \rk(E)\})^2}{2 \rk(F)} \Delta_{H,D}(E).\]
\begin{proof}
The proof is similar to the proof in \cite{CoskunHuizenga2} given in the case of $\PP^2$; we present it for completeness. The object $F$ is a torsion-free sheaf by the standard cohomology sequence and the fact that the heart of the t-structure in the slice we are working in consists of objects which only have nonzero cohomology sheaves in degrees $0$ and $-1$.  The exact sequence along $W$ gives an exact sequence of sheaves
\[ 0\to K\to F \to E \to C \to 0 \]
of ranks $k,f,e,c$, respectively.  By assumption, $k,f,e>0$. Let $(s_W,0)$ be the center of $W$.  As $F$ is in the categories $\cT_{\beta}$ whenever $(\beta,\alpha)$ is on $W$, we find $\mu_{H,D}(F)\geq s_W+\rho_W$,  so 
\begin{align*}
df(s_W+\rho_W) &\leq df \mu_{H,D}(F) = \ch_1^D(F)\cdot H = (\ch_1^D(K)+\ch_1^D(E)-\ch_1^D(C))\cdot H \\& = dk\mu_{H,D}(K) + de\mu_{H,D}(E) - \ch_1^D(C)\cdot H.
\end{align*}
Similarly, $K\in \cF_{\beta}$ along $W$, so $\mu_{H,D}(K)\leq s_W-\rho_W$ and
\[df(s_W+\rho_W) \leq dk(s_W-\rho_W) + de\mu_{H,D}(E) - \ch_1^D(C)\cdot H,\]
which gives
\begin{equation}
\label{eq:ineq-first}
d(k+f) \rho_W \leq d(k-f) s_W +de \mu_{H,D}(E) - \ch_1^D(C)\cdot H.
\end{equation}
We now wish to eliminate the term $\ch_1^D(C)\cdot H$ in Inequality (\ref{eq:ineq-first}). If $C$ is either $0$ or torsion, then $\ch_1^D(C)\cdot H\geq 0$ and $-e = k-f$, and we deduce
\begin{equation}
\label{eq:ineq-wall}
(k+f)\rho_W \leq (k-f)(s_W-\mu_{H,D}(E)).
\end{equation}
Suppose instead that $C$ is not torsion.  Since $C$ is a quotient of the semistable sheaf $E$, we have $\mu_{H,D}(C) \geq \mu_{H,D}(E)$, so $\ch_1^D(C)\cdot H  = dc\mu_{H,D}(C) \geq d c \mu_{H,D}(E)$.  As $k-f = c-e$, we find that Inequality (\ref{eq:ineq-wall}) also holds in this case.

Both sides of Inequality (\ref{eq:ineq-wall}) are positive, so squaring both sides gives
\[(k+f)^2 \rho_W^2 \leq (k-f)^2 (s_W - \mu_{H,D}(E))^2.\]
The formula (\ref{rhoFormula}) for $\rho_W^2$ shows this is equivalent to
\[(k+f)^2 \rho_W^2 \leq (k-f)^2 \left(\rho_W^2 + 2\Delta_{H,D}(E)\right),\]
from which we obtain
\[\rho_W^2 \leq \frac{(k-f)^2}{2kf} \Delta_{H,D}(E).\]
Since $k = f - e + c$, we see that $k \geq \max \{1, f-e \}$.  By taking derivatives in $k$, we see that $\frac{(k-f)^2}{2kf}$ is decreasing for $k+f >0$, and so the maximum possible value of the right-hand side must occur when $k = \max\{1,f-e\}$.  The denominator will be at least $2f$ in this case, and the numerator is $\min \{ (f-1)^2,e^2 \}$.  The result follows.
\end{proof}
\end{lemma}

For our present work we will only need the next consequence of Lemma \ref{lem:higherRank} which follows immediately from computing $\Delta_{H,D}(I_Z)$.

\begin{corollary}\label{cor:walls}
With the hypotheses of Lemma \ref{lem:higherRank}, if $E$ is an ideal sheaf $I_Z\in X^{[n]}$ and $F$ has rank at least $2$, then the radius of the corresponding wall satisfies $$\rho_W^2 \leq \frac{2nd+(H\cdot D)^2-d D^2}{8d^2}:=\varrho_{H,D,n}.$$
\end{corollary}
The number $\varrho_{H,D,n}$ therefore bounds the squares of the radii of higher rank walls for $X^{[n]}$.

\subsection{Rank one walls and critical divisors}  In the cases where we compute the Gieseker wall, the ideal sheaf that is destabilized along the wall will be destabilized by a rank $1$ subobject.  We first compute the numerical walls given by rank $1$ subobjects.

\begin{lemma}\label{lem:rk1Center}
Consider a rank $1$ torsion-free sheaf $F = I_{Z'}(-L)$, where $Z'$ is a zero-dimensional scheme of length $w$ and $L$ is an effective divisor.  In the $(H,D)$-slice, the numerical wall $W$ for $X^{[n]}$ where $F$ has the same slope as an ideal $I_Z$ of $n$ points has center $(s_W,0)$ given by $$s_W = - \frac{2(n-w)+L^2+2(D\cdot L)}{2(H\cdot L)}.$$
\end{lemma}  
\begin{proof}
This is an immediate consequence of Equation (\ref{centerFormula}) for the center of a wall.
\end{proof}

Recalling that walls for $X^{[n]}$ left of the vertical wall get larger as their centers decrease, we deduce the following consequence.

\begin{lemma}
If the Gieseker wall in the $(H,D)$-slice is given by a rank $1$ subobject, then it is a line bundle $\OO_X(-L)$ for some effective divisor $L$.  
\end{lemma}
\begin{proof}
Suppose some $I_Z\in X^{[n]}$ is destabilized along the Gieseker wall $W$ by a sheaf of the form $I_{Z'}(-L)$ where $Z'$ is a nonempty zero-dimensional scheme and $L$ is effective.  By Lemma \ref{lem:rk1Center}, the numerical wall $W'$ given by $\OO_X(-L)$ is strictly larger than $W$.  Since $\OO_X(-L)$ has the same $\mu_{H,D}$-slope as $I_{Z'}(-L)$ and $I_{Z'}(-L)$ is in the categories along $W$, we find that $\OO_X(-L)$ is in at least some of the categories along $W'$.  But then $W'$ is an actual wall, since any ideal sheaf $I_Z$ where $Z$ lies on a curve $C\in |L|$ is destabilized along it.  This contradicts that $W$ is the Gieseker wall.
\end{proof}

Less trivially, there is a further minimality condition automatically satisfied by a line bundle $\OO_X(-L)$ which gives the Gieseker wall.  We define the set of \textit{critical effective divisors} with respect to $H$ and $D$ by
\[\CrDiv(H,D) = \{-D\} \cup \{ L \in \Pic(X) \text{ effective} : H \cdot L < H \cdot (-D) \}. \] By \cite[Ex. V.1.11]{Hartshorne}, the set $\CrDiv(H,D)/{\sim}$ of critical divisors modulo numerical equivalence is finite.  Therefore the set of numerical walls for $X^{[n]}$ given by line bundles $\OO_X(-L)$ with $L\in \CrDiv(H,D)$ is also finite.  Note that the inequality $H\cdot L< H\cdot (-D)$ is equivalent to the inequality $\mu_{H,D}(\OO_X(-L))>0$.  The next proposition demonstrates the importance of critical divisors.

\begin{proposition}\label{prop:critDiv}
Assume $2n > D^2$, and suppose the subobject giving the Gieseker wall for $X^{[n]}$ in the $(H,D)$-slice is a line bundle.  Then the Gieseker wall is computed by $\OO_X(-L)$, where $L\in \CrDiv(H,D)$ is chosen so that the numerical wall given by $\OO_X(-L)$ is as large as possible.
\end{proposition}
\begin{proof}
First, consider the numerical wall $W$ given by $\OO_X(D)$.  By Lemma \ref{lem:rk1Center}, the center $(s_W,0)$ has \begin{equation}\label{Dcenter} s_W = \frac{2n-D^2}{2 (H\cdot D)} < 0\end{equation} since $2n > D^2$ and $D$ is antieffective.  Since $\mu_{H,D}(\OO_X(D)) = \Delta(\OO_X(D)) = 0$, Formula (\ref{rhoFormula}) for the radius of $W$ gives $\rho_W^2 = s_W^2$.  In particular, $W$ is nonempty, and $\OO_X(D)$ lies in at least some of the categories along $W$.  Since $D$ is antieffective, there are exact sequences of the form $$0\to \OO_X(D)\to I_Z \to I_{Z\subset C}\to 0$$ where $C\in |{-D}|$ and $Z\subset C$ is a collection of $n$ points.  If no actual wall is larger than $W$, it follows that $W$ is an actual wall and it is the Gieseker wall.

Suppose the Gieseker wall is larger than $W$ and computed by a line bundle $\OO_X(-L)$ with $L$ effective.  Since $W$ passes through the origin in the $(\beta,\alpha)$-plane, $\OO_X(-L)$ must lie in the category $\cT_0$.  Therefore $\mu_{H,D}(\OO_X(-L)) > 0$, and $L\in \CrDiv(H,D)$.

Conversely, suppose $L\in \CrDiv(H,D)$ is chosen to maximize the wall $W'$ given by $\OO_X(-L)$.  Then no actual wall is larger than $W'$.  Since $s_W <0$ and $\mu_{H,D}(\OO_X(-L)) \geq 0$, we find that $\OO_X(-L)$ is in at least some of the categories along $W$, and hence in at least some of the categories along $W'$.  We conclude that $W'$ is an actual wall, and therefore that it is the Gieseker wall.
\end{proof}

Combining Corollary \ref{cor:walls} and Proposition \ref{prop:critDiv} gives our primary tool to compute the Gieseker wall.

\begin{theorem}\label{thm:Gieseker}
Assume $2n > D^2$, and let $L\in \CrDiv(H,D)$ be a critical divisor such that the wall for $X^{[n]}$ given by $\OO_X(-L)$ is as large as possible.  If this wall has radius $\rho$ satisfying $\rho^2 \geq \varrho_{H,D,n}$, then it is the Gieseker wall.

Conversely, if the Gieseker wall has radius satisfying $\rho^2 \geq \varrho_{H,D,n}$  then it is obtained in this way.
\end{theorem}

While the theorem is our sharpest result, it is useful to lose some generality to get a more explicit version.  Since $-D\in \CrDiv(H,D)$, if the wall given by $\OO_X(D)$ satisfies $\rho^2 \geq \varrho_{H,D,n}$ then the Gieseker wall is computed by Theorem \ref{thm:Gieseker}.  This allows us to compute the Gieseker wall so long as $n$ is large enough, depending only on the intersection numbers of $H$ and $D$.

\begin{corollary}
\label{cor:Dbest}
Let $$\eta_{H,D}: = \frac{(H\cdot D)^2+d D^2}{2d}.$$ If $n\geq \eta_{H,D}$ then the Gieseker wall is the largest wall given by a critical divisor.

Furthermore, if $n>\eta_{H,D}$ then every $I_Z$ destabilized along the Gieseker wall fits into an exact sequence $$0 \to \OO_X(-C)\to I_Z\to I_{Z\subset C}\to 0$$ for some curve $C\in |L|$, where $L$ is a critical divisor computing the Gieseker wall.  If the critical divisor computing the Gieseker wall is unique, then $\OO_X(-C)$ and $I_{Z\subset C}$ are the Jordan--H\"older factors of any $I_Z$ destabilized along the Gieseker wall.
\end{corollary}
\begin{proof}
Observe that the inequality $n\geq \eta_{H,D}$ automatically implies the inequality $2n>D^2$ needed to apply Theorem \ref{thm:Gieseker}.

Let $W$ be the wall for $X^{[n]}$ in the $(H,D)$-slice corresponding to $\OO_{X}(D)$.  The center $(s_W,0)$ of $W$ was computed in Equation (\ref{Dcenter}), and $\rho_W^2=s_W^2$.  We find that $\rho_W^2\geq \varrho_{H,D,n}$ holds when $n\geq \eta_{H,D}$, with strict inequality when $n>\eta_{H,D}$. 

When $n >\eta_{H,D}$ there can be no higher-rank destabilizing subobject of an $I_Z$ destabilized along the Gieseker wall, so there is an exact sequence as claimed.  Furthermore, if there is only one critical divisor computing the wall, then there is a unique destabilizing subobject along the wall, so the Jordan--H\"older filtration has length two.
\end{proof}

\subsection{Classes of divisors} In this subsection we give an elementary computation of the class of the divisor corresponding to a wall in a given slice of the stability manifold.  Similar results have been obtained by Liu \cite{Liu}, but the result is critical to our discussion so we include the proof.  See \cite[\S4]{BayerMacri} for more details on the definitions and results we use here. 

Throughout this subsection, let ${\bf v}\in K_{0}(X)$ be a vector such that the moduli space $M_{H,D}({\bf v})$ of $(H,D)$-Gieseker semistable sheaves admits a (quasi-)universal family $\cE$ which is unique up to equivalence (Hilbert schemes $X^{[n]}$ are examples of such spaces).  We also let $\sigma = (Z,\cA)$ be a stability condition in the closure of the Gieseker chamber for ${\bf v}$ in the $(H,D)$-slice.  Then there is a well-defined corresponding divisor $D_{\sigma}\in N^1(M_{H,D-K_X/2}({\bf v}))$ which is independent of the choice of $\cE$.  

Let $({\bf v},{\bf w}) = \chi({\bf v}\cdot {\bf w})$ be the Euler pairing on $K_{\num}(X)_\R$, and write ${\bf v}^\perp \subset K_{\num}(X)_\R$ for the orthogonal complement with respect to this pairing.  The correspondence between stability conditions and divisor classes is understood in terms of the Donaldson homomorphism $$\lambda: {\bf v}^\perp \to N^1(M_{H,D-K_X/2}({\bf v})).$$ Since the Euler pairing is nondegenerate, there is a unique vector $ {\bf w}_\sigma\in {\bf v}^\perp$ such that $$\Im\left(-\frac{Z({\bf w}')}{Z({\bf v})}\right) = ({\bf w'},{\bf w_\sigma})$$ for all ${\bf w}'\in K_{\num}(X)_\R$.  Bayer and Macr\`i show that $D_\sigma = \lambda({\bf w}_\sigma)$.  In what follows, we write vectors in $K_{\num}(X)_\R$ as $(\ch_0,\ch_1,\ch_2)$.

\begin{proposition}
\label{prop:classComputation}
With the above assumptions, suppose $\sigma$ lies on a numerical wall $W$ in the $(H,D)$-slice with center $(s_W,0)$.  Then ${\bf w}_\sigma$ is a multiple of $$(-1,-\frac{1}{2}K_X + s_W H +D,m)\in {\bf v}^\perp,$$  where $m$ is determined by the requirement ${\bf w}_\sigma\in {\bf v}^\perp$.

In particular, if $X$ has irregularity $0$ and ${\bf v} = (1,0,-n)$ is the vector for $X^{[n]}$, then the divisor $D_\sigma$ is a multiple of $$\frac{1}{2}K_X^{[n]}-s_W H^{[n]}-D^{[n]} - \frac{1}{2}B.$$
\end{proposition}

\begin{remark}
Suppose $X$ has irregularity $0$.  Up to scale, the divisors induced by stability conditions in the $(H,D)$-slice give a ray in $N^1(X^{[n]})$ emanating from the class $H^{[n]}\in \Nef(X)\subset \Nef(X^{[n]})$.  The particular ray is determined by the choice of the twisting divisor $D$.
\end{remark}

\begin{proof}[Proof of Proposition \ref{prop:classComputation}]
Since $\sigma$ is in the $(H,D)$-slice, write $\sigma = \sigma_{\beta,\alpha}$ and $(Z,\cA) = (Z_{\beta,\alpha},\cA_\beta)$ for short.  Put $z = -1/Z({\bf v}) = u+iv$.  We evaluate the identity $$\Im(z Z({\bf w}')) = ({\bf w'}, {\bf w}_\sigma)$$ defining ${\bf w}_\sigma$ on various classes ${\bf w'}$ to compute ${\bf w}_\sigma$.

Write the Chern character ${\bf w}_\sigma = (r,C,d)$.  Then $$-v = \Im(z Z(0,0,1)) =((0,0,1),{\bf w}_\sigma)=r,$$ so $r = -v$.  Next, for any curve class $C'$, $$(u+\beta v)(C'\cdot H) + v(C'\cdot D)=\Im(z Z(0,C',0))=((0,C',0),{\bf w}_\sigma) = \chi(0,-vC',C'\cdot C).$$ By Riemann--Roch and adjunction, $$\chi(0,-vC',C'\cdot C) = -v\left( \left( -\frac{1}{v} (C'\cdot C)+ \frac{1}{2} (C')^2\right)-\frac{1}{2}(C')^2-\frac{1}{2}(C'\cdot K_X)\right)=C'\cdot C+\frac{v}{2}(K_X\cdot C'),$$ so $$C'\cdot C = (u+\beta v)(C'\cdot H)+v(C'\cdot D)-\frac{v}{2}(C'\cdot K_X)$$ for every class $C'$.  Thus for any class $C'$ with $C'\cdot H = 0$, we have $C'\cdot C = v(C'\cdot D)-\frac{v}{2}(C'\cdot K_X);$ it follows that there is some number $a$ with $$C =  - \frac{v}{2}K_X + aH+ vD.$$ Considering $C = H$ shows that $a = u+\beta v$.  Therefore $${\bf w}_\sigma = (-v,-\frac{v}{2}K_X + (u+\beta v)H,m), $$ where $m$ is chosen such that ${\bf w}_\sigma\in {\bf v}^\perp$.

Finally, a straightforward calculation shows that
\[\frac{u}{v} + \beta = \nu_\sigma({\bf v})+ \beta = s_W\]
holds for all $(\beta,\alpha)$ along $W$.  The follow up statement for Hilbert schemes follows by computing the Donaldson homomorphism.
\end{proof}

\subsection{Dual curves}

Suppose $D_{\sigma_0}$ is the nef divisor corresponding to the Gieseker wall for $X^{[n]}$ in the $(H,D)$-slice.  Showing that $D_{\sigma_0}$ is an extremal nef divisor amounts to showing that there is some curve $\gamma\subset X^{[n]}$ with $D_{\sigma_0}\cdot \gamma=0$.  By the Positivity Lemma, this happens when $\gamma$ parameterizes objects of $X^{[n]}$ which are generically $S$-equivalent with respect to $\sigma_0$.  

In every case where we computed the Gieseker wall, the wall can be given by a destabilizing subobject which is a line bundle $\OO_X(-C)$ with $C$ an effective curve.  If $Z$ is a length $n$ subscheme of $C$, then there is a destabilizing sequence $$0\to \OO_X(-C)\to I_Z \to I_{Z\subset C}\to 0.$$ If $\ext^1(I_{Z\subset C},\OO_X(-C)) \geq 2$, then curves of objects of $X^{[n]}$ which are generically $S$-equivalent with respect to $\sigma_0$ are obtained by varying the extension class.  We obtain the following general result.

\begin{lemma}\label{lem:curvesExist}
Suppose the Gieseker wall for $X^{[n]}$ in the $(H,D)$-slice is computed by the subobject $\OO_X(-C)$, where $C$ is an effective curve class of arithmetic genus $p_a(C)$.  If $n\geq p_a(C)+1$, then the corresponding nef divisor $D_{\sigma_0}$ is extremal.
\end{lemma}
\begin{proof}
Bilinearity of the Euler characteristic $\chi(\cdot,\cdot)$ and Serre duality shows that $$\chi(I_{Z\subset C},\OO_X(-C))= p_a(C) - 1 - n.$$ Therefore, once $n\geq p_a(C)+1$ we will have $\chi(I_{Z\subset C},\OO_X(-C))\leq -2$, and curves orthogonal to $D_{\sigma_0}$ can be constructed by varying the extension class.
\end{proof}

Combining Lemma \ref{lem:curvesExist} with our previous results on the computation of the Gieseker wall gives us the following asymptotic result.

\begin{theorem}\label{thm:asymptotic}
Fix a slice $(H,D)$ for $\Stab(X)$.  There is some $L\in \CrDiv(H,D)$ such that for all $n \gg 0$ the Gieseker wall is computed by $\OO_X(-L)$. Furthermore, the corresponding nef divisor is extremal.    
\end{theorem}
\begin{proof}
Recall that the set $\CrDiv(H,D)/{\sim}$ of critical divisors modulo numerical equivalence is finite; say $\{L_1,\ldots,L_m\}$ is a set of representatives.  For $1\leq i \leq m$, let $(s_{i}(n),0)$ be the center of the wall $\OO_X(-L_i)$ for $X^{[n]}$.  Then $s_i(n)$ is a linear function of $n$ by Lemma \ref{lem:rk1Center}, so there is some $i$ with $s_i(n)\leq s_j(n)$ for all $1\leq j\leq m$ and $n\gg 0$.  Then by  Corollary \ref{cor:Dbest} the Gieseker wall is given by $\OO_X(-L_i)$.  Again increasing $n$ if necessary, the divisor $D_{\sigma_0}$ corresponding to the Gieseker wall is extremal by Lemma \ref{lem:curvesExist}.
\end{proof}

\begin{remark}The requirement $n\geq p_a(C) +1$ in Lemma \ref{lem:curvesExist} is not typically sharp.  For example, if $|C|$ contains a smooth curve we may as well assume $C$ is smooth.  Then $I_{Z\subset C}$ is a line bundle on $C$, and $$\Ext^1(I_{Z\subset C},\OO_X(-C)) \cong H^0(\OO_C(Z)).$$ Thus $g_n^1$'s on $C$ give curves which are orthogonal to $D_{\sigma_0}$.  The following fact from Brill--Noether theory therefore provides curves on $X^{[n]}$ for smaller values on $n$.

\begin{lemma}\cite{ACGH}
If $C$ is smooth of genus $g$, then it has a $g_n^1$ for any $n\geq \lceil{\frac{g+2}{2}}\rceil$.
\end{lemma}

For specific surfaces, some curves in $|C|$ may have highly special linear series giving better constructions of curves on $X^{[n]}$.
\end{remark}

\section{Picard rank one examples}
\label{sec-generalType}

For the rest of the paper, we will apply the methods of Section \ref{sec-nef} to compute $\Nef(X^{[n]})$ for several interesting surfaces $X$.    These applications form the heart of the paper.

\subsection{Picard rank one in general}\label{ssec:picRank1} Suppose $\Pic(X) \cong \Z H$ for some ample divisor $H$.  If we choose $D = -aH$, where $a>0$ is the smallest positive integer such that $aH$ is effective, then $\CrDiv(H,D) = \{-D\}$.

\begin{lemma}\label{lem:picRank1}
Suppose $\Pic(X) = \Z  H$ and $aH$ is the minimal effective class.  If $n\geq (aH)^2 = a^2 d$, then the Gieseker wall for $X^{[n]}$ is the wall given by $\OO_X(-aH)$.
\end{lemma} 
\begin{proof}
Apply Corollary \ref{cor:Dbest} with $D = -aH$.
\end{proof}

Note that when $n>a^2d$, additional information about the Jordan--H\"older filtration can be obtained as in Corollary \ref{cor:Dbest}.  We use Formula (\ref{Dcenter}) to see that the wall $W$ given by $\OO_X(-aH)$ has center $(s_W,0)$ with $$s_W = \frac{a}{2}-\frac{n}{ad}$$ Combining Lemmas \ref{lem:picRank1}, \ref{lem:curvesExist}, and Proposition \ref{prop:classComputation}, we have proved the following general result.

\begin{theorem}\label{thm:picRank1}
Suppose $\Pic X \cong \Z H$ and $aH$ is the minimal effective class.  If $n\geq a^2d$ then the divisor \begin{equation}\label{picRank1Div}\frac{1}{2}K_X^{[n]}+\left(\frac{a}{2}+\frac{n}{ad}\right)H^{[n]}-\frac{1}{2}B\end{equation} is nef. Additionally, if $n\geq p_a(aH)+1$ then this divisor is extremal, so $\Nef(X^{[n]})$ is spanned by this divisor and $H^{[n]}$. An orthogonal curve is given by letting $n$ points move in a $g_n^1$ on a curve of class $aH$.\end{theorem}

\begin{remark} If $\Pic(X) = \Z H$ and $H$ is already effective, then a different argument computes the Gieseker wall so long as $2n > d$, improving the bound in Lemma \ref{lem:picRank1}.  However, fine information about the Jordan--H\"older filtration of a destabilized ideal sheaf is not obtained.  In fact, if $n\leq d$ then the destabilizing behavior can be complicated.  For instance, a scheme $Z$ contained in the complete intersection of two curves of class $H$ will admit an interesting map from $\OO_X(-H)^{\oplus 2}$.
\end{remark}
\begin{proposition}
\label{prop:2nBound}
Suppose $\Pic X =\Z H$ and $H$ is effective. If $2n > d$, then the Gieseker wall for $X^{[n]}$ in the $(H,-H)$-slice is the wall given by $\OO_X(-H)$.  Thus the divisor (\ref{picRank1Div}) with $a=1$ is nef.
\end{proposition}
\begin{proof}
Let $W$ be the numerical wall given by $\OO_X(-H)$.
By the proof of Proposition \ref{prop:critDiv}, if no actual wall is larger than $W$ then $W$ is an actual wall, and hence the Gieseker wall.
If there is a destabilizing sequence $$0\to F\to I_Z \to G\to 0$$ giving a wall $W'$ larger than $W$, then $F,G\in \cA_0$ since $W$ passes through the origin in the $(\beta,\alpha)$-plane.  Fix $\alpha>0$ such that $(0,\alpha)$ lies on $W'$. We have $$H \cdot \ch_1^{-H}(F) = \Im Z_{0,\alpha}(F)\geq 0 \quad \mbox{and}\quad H\cdot \ch_1^{-H}(G) = \Im Z_{0,\alpha}(G) \geq 0.$$ Since $d$ is the smallest intersection number of $H$ with an integral divisor and $$d=\Im Z_{0,\alpha}(I_Z) = \Im Z_{0,\alpha}(F)+\Im Z_{0,\alpha}(G)$$ we conclude that either $\Im Z_{0,\alpha}(F)=0$ or $\Im Z_{0,\alpha}(G) = 0$.  Thus either $F$ or $G$ has infinite $\sigma_{0,\alpha}$-slope, contradicting that $(0,\alpha)$ is on $W'$.
\end{proof}

We now further relax the lower bound on $n$ needed to guarantee the existence of orthogonal curve classes in special cases.

\subsection{Surfaces in $\PP^3$} By the Noether--Lefschetz theorem, a very general surface $X\subset \PP^3$ of degree $d\geq 4$ is smooth of Picard rank $1$ and irregularity 0. Let $H$ be the hyperplane class and put $D = -H$. We have $K_X = (d-4)H$, so Proposition \ref{prop:2nBound} shows that if $2n>d$ then the divisor $$\left(\frac{d}{2}-\frac{3}{2}+\frac{n}{d}\right)H^{[n]}-\frac{1}{2}B$$ is nef.  If $C$ is any smooth hyperplane section then the projection from a point on $C$ gives a degree $d-1$ map to $\PP^1$, so $C$ carries a $g_n^1$ for any $n\geq d-1$.  We have proved the following result.

\begin{proposition}\label{prop:P3}
Let $X$ be a smooth degree $d$ hypersurface in $\PP^3$ with Picard rank $1$.  The divisor $$\left(\frac{d}{2}-\frac{3}{2}+\frac{n}{d}\right)H^{[n]} - \frac{1}{2}B$$ on $X^{[n]}$ is nef if $2n>d$.  If $n\geq d-1$, then it is extremal, and together with $H^{[n]}$ it spans $\Nef(X^{[n]})$.
\end{proposition} 

\begin{remark}
The behavior of $\Nef(X^{[n]})$ for smaller $n$ in Proposition \ref{prop:P3} is more mysterious.  Even the cases $d=5$ and $n=2,3$ are interesting.
\end{remark}

\begin{remark}
The case $d=4$ of Proposition \ref{prop:P3} recovers a special case of \cite[Proposition 10.3]{BayerMacri} for K3 surfaces.  The case $d=1$ recovers the computation of the nef cone of $\PP^{2[n]}$ \cite{ABCH}.
\end{remark}

\subsection{Branched covers of $\PP^2$}  Next we consider cyclic branched covers of $\PP^2$.  Let $X$ be a very general cyclic degree $d$ cover of $\bP^2$, branched along a degree $e$ curve.  Note that this means that $d$ necessarily divides $e$.  We can view these covers as hypersurfaces in a weighted projective space, which gives us a Noether--Lefschetz type theorem: $\Pic X = \Z H$, generated by the pullback $H$ of the hyperplane class on $\PP^2$, provided that $X$ has positive geometric genus. The canonical bundle of $X$ is $$K_X = -3H + e\left(\frac{d-1}{d}\right)H = \left(\frac{e(d-1)}{d}-3\right)H.$$  Then $X$ will have positive geometric genus if $e \geq 3d/(d-1)$.

Setting $D = -H$, we see that if $2n > d$ then the divisor class  $$\left(\frac{e(d-1)}{2d}-1 +\frac{n}{d}\right)H^{[n]}-\frac{1}{2}B$$ is nef by Proposition \ref{prop:2nBound}. The preimage of a line is a curve of class $H$, and it carries a $g_d^1$ given by the map to $\PP^2$.  Therefore the above divisor is extremal once $n\geq d$.

\begin{proposition}\label{prop:cyclic}
Let $X$ be a very general degree $d$  cyclic cover of $\bP^2$ ramified along a degree $e$ curve, where $d$ divides $e$ and $e \geq \frac{3d}{d-1}$. The divisor
$$\left(\frac{e(d-1)}{2d}-1 +\frac{n}{d}\right)H^{[n]}-\frac{1}{2}B$$ on $X^{[n]}$ is nef if $2n > d$.
For $n \geq d$, this class is extremal, and together with $H^{[n]}$ it spans $\Nef(X^{[n]})$.
\end{proposition}

\section{Del Pezzo surfaces of degree one}
\label{sec-delPezzo}

In \cite{BertramCoskun}, Bertram and Coskun studied the birational geometry of $X^{[n]}$ when $X$ is a minimal rational surface or a del Pezzo surface.  In particular, they completely computed the nef cones of all these Hilbert schemes except in the case of a del Pezzo surface of degree $1$.  The constructions they gave were classical: they produced nef divisors from $k$-very ample line bundles, and dual curves by letting collections of points move in linear pencils on special curves.  

In this section, we will compute the nef cone of $X^{[n]}$, where $X$ is a smooth del Pezzo surface of degree $1$.  Then $X \cong \Bl_{p_1,\ldots,p_8} \PP^2$ for distinct points $p_1,\ldots,p_8$ with the property that $-K_X$ is ample (see \cite[Theorem 24.4]{Manin} or \cite[Ex. V.21.1]{Beauville}).  This application exhibits the full strength of the methods of Section \ref{sec-nef}.

\subsection{Notation and statement of results} 
Let $H$ be the class of a line and let $E_1,\ldots,E_8$ be the $8$ exceptional divisors over the $p_i$, so $\Pic(X) \cong \Z H \oplus \Z E_1\oplus \cdots \oplus \Z E_8$ and $K_X = -3H + \sum_i E_i$.  Recall that a \emph{$(-1)$-curve} on $X$ is a smooth rational curve of self-intersection $-1$.  It is simplest to describe the dual cone of effective curves.  We recommend reviewing \S\ref{ssec-divisorsAndCurves} for notation.

\begin{theorem}\label{thm-delPezzoCurves}
The cone of curves $\NE(X^{[n]})$ is spanned by all the classes $E_{[n]}$ given by $(-1)$-curves $E\subset X$, the class of a curve contracted by the Hilbert--Chow morphism, and the class $F_{[n]}$, where $F\in |{-K_X}|$ is an anticanonical curve. 
\end{theorem}

 The $240$ $(-1)$-curves $E$ on $X$ are well-known. The possible classes are $$(0;1) \quad (1; 1^2) \quad (2; 1^5) \quad (3;2,1^6) \quad (4; 2^3,1^5) \quad (5;2^6,1^2) \quad (6;3,2^7),$$ where e.g. $(4;2^3,1^5)$ denotes any class equivalent to $$4H - 2E_1-2E_2-2 E_3 -E_4-E_5-E_6-E_7-E_8$$ under the natural action of $S_8$ on $\Pic(X)$.  The cone of curves $\NE(X)$ is spanned by the classes of the $(-1)$-curves.  The Weyl group action on $\Pic(X)$ acts transitively on $(-1)$-curve classes, so it acts transitively on the extremal rays of $\NE(X)$.  We refer the reader to \cite[\S 26]{Manin} for details.  On the other hand, there are two orbits of extremal rays of $\Nef(X)$.
 
 \begin{proposition}
 For $k\geq 2$, let $X_k$ be the del Pezzo surface $X_k \cong \Bl_{p_1,\ldots,p_k} \PP^2$, where $p_1,\ldots,p_k$ are distinct points such that $-K_{X_k}$ is ample.  Then the Weyl group acts on the extremal rays of $\Nef(X_k)$, and there are two orbits.  The classes $H$ and $H-E_1$ are orbit representatives.
\end{proposition}
\begin{proof}
Since the Weyl group preserves the intersection pairing and $H$ and $H-E_1$ are extremal nef divisors with different self-intersections, there are at least two orbits.

For $k=2$, the nef cone is spanned by $H, H-E_1, H-E_2$.  The Weyl group $\Z/2\Z$ fixes $H$ and exchanges $H-E_1$ and $H-E_2$, so there are two orbits.

Now suppose $k> 2$ and $N$ is an extremal nef divisor on $X_k$.  Then $N$ is orthogonal to a face of $\NE(X)$, so there is a $(-1)$-curve orthogonal to $N$.  Since $k>2$, we may use the Weyl group to assume this $(-1)$-curve is $E_k$.  Then $N$ is a pullback $N = \pi^* N'$ along the blowdown map $\pi : X_k \to X_{k-1}$ contracting $E_k$.  Since $N$ is an extremal nef divisor on $X_k$, $N'$ is an extremal nef divisor on $X_{k-1}$: a nontrivial decomposition $N' = A + B$ with $A,B$ nef would pullback to a nontrivial decomposition of $N$.  Continuing in this fashion we see that up to the action of the Weyl group $N$ is the pullback of $H$ or $H-E_1$ from $X_2$.
\end{proof}

Consider the divisor class $(n-1)(-K_X)^{[n]} - \frac{B}{2}$.  If $E$ is any $(-1)$-curve on $X$, then $-K_X \cdot E = 1$, so $$E_{[n]}\cdot ((n-1)(-K_X)^{[n]} - \frac{1}{2}B) = (n-1)(-K_X\cdot E)-(n-1)=0.$$ Let $\Lambda\subset N^1(X^{[n]})$ be the cone spanned by divisors which are nonnegative on all classes $E_{[n]}$ and curves contracted by the Hilbert--Chow morphism.  It follows that $\Lambda \supset \Nef(X^{[n]})$ is spanned by $\Nef(X)\subset \Nef(X^{[n]})$ and the single additional class $(n-1)(-K_X)^{[n]}-\frac{B}{2}$.

However, $\Nef(X^{[n]})\subset \Lambda$ is a \emph{proper} subcone.  Indeed, if $F\in |{-K_X}|$ is an anticanonical curve then by Riemann-Hurwitz $F_{[n]}\cdot B = 2n$, so  $F_{[n]}\cdot ((n-1)(-K_X)^{[n]}-\frac{B}{2}) = -1.$  Let $\Lambda'\subset \Lambda$ be the subcone of $F_{[n]}$-nonnegative divisors.  Taking duals, we see that Theorem \ref{thm-delPezzoCurves} is equivalent to the next result.

\begin{theorem}\label{thm-delPezzoNef}
We have $\Nef(X^{[n]}) = \Lambda'$.
\end{theorem}

To prove Theorem \ref{thm-delPezzoNef}, we must show that all the extremal rays of $\Lambda'$ are actually nef. Suppose $N\in \Nef(X)$ spans an extremal ray of $\Nef(X)$.  Then the cone spanned by $N^{[n]}$ and $(n-1)(-K_X)^{[n]}-\frac{B}{2}$ contains a single ray of $F_{[n]}$-orthogonal divisors, and this ray is an extremal ray of $\Lambda'$. Conversely, due to our description of the cone $\Lambda$, the extremal rays of $\Lambda'$ which are not in $\Nef(X)$ are all obtained in this way.

\subsection{Choosing a slice} More concretely, making use of the Weyl group action we may as well assume our extremal nef class $N\in \Nef(X)$ is either $H-E_1$ or $H$.  The corresponding $F_{[n]}$-orthogonal rays described in the previous paragraph are spanned by \begin{equation}\label{divClass1}(n-1)(-K_X)^{[n]}+\frac{1}{2}(H^{[n]}-E_1^{[n]})  -\frac{1}{2}B\end{equation} and \begin{equation}\label{divClass2}(n-1)(-K_X)^{[n]}+\frac{1}{3}H^{[n]}  -\frac{1}{2}B,\end{equation} respectively.  Our job is to show that these two classes are nef.  We will prove this by exhibiting these divisors as nef divisors on $X^{[n]}$ corresponding to the Gieseker wall for suitable choices of slices of $\Stab(X)$.

To show that the divisor (\ref{divClass1}) is nef, it is convenient to choose our polarization to be $$P=\left(n-\frac{3}{2}\right)(-K_X)+\frac{1}{2}(H-E_1)$$ (which depends on $n$!) and our antieffective class to be $D = K_X$.  Observe that $P$ is ample since it is the sum of an ample and a nef class.  If we show that the Gieseker wall $W$ in the $(P,K_X)$-slice has center $(s_W,0)=(-1,0)$, then Proposition \ref{prop:classComputation} implies the divisor class (\ref{divClass1}) is nef.  

Similarly, to show that the divisor (\ref{divClass2}) is nef, we choose the  polarization $$Q = \left(n-\frac{3}{2}\right)(-K_X)+\frac{1}{3}H$$ and antieffective divisor $D = K_X$.  Again, by Proposition \ref{prop:classComputation} we must show the Gieseker wall $W$ in the $(Q,K_X)$ slice has center $(s_W,0) = (-1,0)$. 

\subsection{Critical divisors} Our plan is to apply Corollary \ref{cor:Dbest} to compute the Gieseker wall in the $(P,K_X)$- and $(Q,K_X)$-slices.  We must first identify the set $\CrDiv(P,K_X)$ of critical divisors.  

\begin{lemma}\label{lem:dPcritical}
If $n> 2$, then the set $\CrDiv(P,K_X)$ consists of $-K_X$ and the classes $L$ of $(-1)$-curves on $X$ with $L\cdot (H-E_1) \leq 1$.

When $n=2$, the above classes are still critical.  Additionally, the class $H-E_1$ is critical, as is any sum of two $(-1)$-curves $L_1,L_2$ with $L_i\cdot (H-E_1)=0$.
\end{lemma}
\begin{proof}
Write $2P = A + N$ where $A = (2n-3)(-K_X)$ is ample and $N = H-E_1$ is nef.  Then $A\cdot (-K_X) = 2n-3$ and $N \cdot (-K_X) = 2$, so an effective curve class $L\neq -K_X$ is in $\CrDiv(P,K_X)$ if and only if $L\cdot (2P) < 2n-1$.

First suppose $n>2$, and let $L\in \CrDiv(P,K_X)$.  If $L\cdot (-K_X) \geq 2$, then $L\cdot(2P) \geq 4n-6>2n-1$, so $L$ is not critical.  Therefore $L\cdot (-K_X) = 1$.  Thus any curve of class $L$ is reduced and irreducible.  By the Hodge index theorem, $$L^2 = L^2\cdot (-K_X)^2 \leq (L\cdot (-K_X))^2 = 1,$$ with equality if and only if $L = -K_X$.  If the inequality is strict, then by adjunction we must have $L^2 = -1$ and $L$ is a $(-1)$-curve. Since $L\cdot (2P)< 2n-1$, we further have $L \cdot N \leq 1$.
 
Suppose instead that $n=2$ and $L\in \CrDiv(P,K_X)$. The cases $L \cdot (-K_X) \leq  1$ and $L \cdot (-K_X) \geq  3$ follow as in the previous case. The only other possibility is that $L\cdot (-K_X) = 2$ and $L\cdot N=0$.  Since $L\cdot N=0$, the curve $L$ is a sum of curves in fibers of the projection $X\to \PP^1$ given by $|N|$.  This easily implies the result.
\end{proof}

An essentially identical computation computes $\CrDiv(Q,K_X)$.  We omit the proof.

\begin{lemma}
If $n>2$, then the set $\CrDiv(Q,K_X)$ consists of $-K_X$ and the classes $L$ of $(-1)$-curves on $X$ with $L\cdot H \leq 2$.

When $n=2$, the above classes are still critical.  Additionally, for $1\leq i,j\leq 8$ the classes $H-E_i$ and $E_i+E_j$ are also critical.
\end{lemma}

The next application of Corollary \ref{cor:Dbest} shows the divisor class (\ref{divClass1}) is nef.
\begin{proposition}
The Gieseker wall for $X^{[n]}$ in the $(P,K_X)$-slice has center $(-1,0)$, and is given by the subobject $\OO_X(K_X)$.  It coincides with the wall given by $\OO_X(-L)$, where $L$ is any $(-1)$-curve with $L\cdot (H-E_1)=0$.  Therefore, the divisor $$(n-1)(-K_X)^{[n]}+\frac{1}{2}(H^{[n]}-E_1^{[n]})  -\frac{1}{2}B$$ is nef.
\end{proposition}
\begin{proof}
By Equation (\ref{Dcenter}), the center of the wall for $\OO_X(K_X)$ is $(s_W,0)$ with $$s_W = \frac{2n-K_X^2}{(2P)\cdot K_X} = -1.$$ A straightforward computation shows $\eta_{P,K_X} < n$ for all $n\geq 2$.  Therefore, by Corollary \ref{cor:Dbest}, the Gieseker wall is computed by a critical divisor.

We only need to verify that no other critical divisor gives a larger wall. Let $L\in \CrDiv(P,K_X)$. By Lemma \ref{lem:rk1Center}, the center of the wall given by $\OO_X(-L)$ lies at the point $(s_L,0)$ where $$s_L = - \frac{2n+L^2+2(K_X\cdot L)}{(2P)\cdot L}.$$ If $L$ is a $(-1)$-curve, then $$s_L = -\frac{2n-3}{(2P)\cdot L} =-\frac{2n-3}{2n-3+L\cdot(H-E_1)}\geq -1,$$ with equality if and only if $L\cdot (H-E_1)=0$.  This proves the result if $n>2$.

To complete the proof when $n=2$, we only need to consider the additional critical classes mentioned in Lemma \ref{lem:dPcritical}.  For every such $L\in \CrDiv(P,K_X)$ we have $L\cdot K_X = -2$ and $L^2\leq 0$.  Thus $s_L\geq 0$ for every such divisor.
\end{proof}

Finally, an identical computation for the $(Q,K_X)$-slice shows the divisor class (\ref{divClass2}) is nef.  This completes the proof of Theorems \ref{thm-delPezzoCurves} and \ref{thm-delPezzoNef}.

\begin{proposition}
The Gieseker wall for $X^{[n]}$ in the $(Q,K_X)$-slice has center $(-1,0)$, and is given by the subobject $\OO_X(K_X)$.  It coincides with the wall given by $\OO_X(-E_i)$ for any $1\leq i \leq 8$.  Therefore, the divisor $$(n-1)(-K_X)^{[n]}+\frac{1}{2}(H^{[n]}-E_1^{[n]})  -\frac{1}{2}B$$ is nef.
\end{proposition}

\bibliographystyle{plain}

\end{document}